\documentclass[11pt]{article}

\input xy

\xyoption{all}

\usepackage{amssymb,amsbsy,amsthm,amsmath,mathrsfs,latexsym,graphicx,epsfig}
\usepackage{mathabx,times}

\usepackage{sectsty}

\sectionfont{\large}

\subsectionfont{\normalsize}

\newtheorem{thm}{Theorem}
\newtheorem{lem}[thm]{Lemma}

\newtheorem{cor}[thm]{Corollary}

\newtheorem{dfn}[thm]{Definition}

\newtheorem{ques}[thm]{Question}
\theoremstyle{remark}
\newtheorem{ex}{Example}

\renewcommand{\rm}[1]{\mathrm{#1}}

\newcommand{\bbC}{\mathbb{C}}

\newcommand{\bbK}{K}
\newcommand{\bbN}{\mathbb{N}}

\newcommand{\bbR}{\mathbb{R}}

\newcommand{\bbZ}{\mathbb{Z}}

\newcommand{\rmH}{\mathrm{H}}

\newcommand{\B}{\mathcal{B}}
\newcommand{\C}{\mathcal{C}}

\newcommand{\Z}{\mathcal{Z}}

\renewcommand{\a}{\alpha}
\renewcommand{\b}{\beta}

\newcommand{\ol}[1]{\overline{#1}}

\renewcommand{\t}[1]{\tilde{#1}}

\newcommand{\fin}{\nolinebreak\hspace{\stretch{1}}$\lhd$}

\newcommand{\actson}{\curvearrowright}
\renewcommand{\to}{\longrightarrow}

\begin{document}

\title{Euclidean-valued group cohomology is always reduced}

\author{Tim Austin\footnote{Research supported by a fellowship from the Clay Mathematics Institute}}

\date{}

\maketitle

\begin{abstract}
Let $K$ be a topological field and $G$ a countable discrete group.  Then for any linear $G$-action on a finite-dimensional vector space over $K$, the groups of coboundaries in the inhomogeneous bar resolution are closed in all degrees, and hence the cohomology is reduced in all degrees. This can be deduced from a more general automatic-closure theorem for continuous linear transformations between inverse limits of finite-dimensional vector spaces.

\vspace{7pt}


\vspace{7pt}

\end{abstract}

Let $G$ be a countable discrete group and $\pi:G\actson M$ a continuous action on a topological Abelian group.  A fundamental tool for describing the structure of this $G$-module is the group cohomology $\rmH^\ast(G,M)$.  This is often defined in terms of the inhomogeneous bar resolution
\[M\stackrel{d_1}{\to} \C^1(G,M) \stackrel{d_2}{\to} \C^2(G,M) \stackrel{d_3}{\to}\cdots,\]
where $\C^p(G,M)$ is the group of all functions $G^p\to M$, and $d_p:\C^{p-1}(G,M)\to \C^p(G,M)$ is the coboundary operator defined by
\begin{multline*}
d_pf(g_1,\ldots,d_p) := \pi^{g_1}f(g_2,\ldots,g_p) + \sum_{i=1}^p(-1)^if(g_1,\ldots,g_ig_{i+1},\ldots,g_p)\\ + (-1)^{p-1}f(g_1,\ldots,g_{p-1}).
\end{multline*}
In terms of these objects, one sets
\begin{multline*}
\Z^p(G,M) := \ker d_{p+1}, \quad \B^p(G,M) := \rm{img}\,d_p, \\ \hbox{and}\quad \rmH^p(G,M) := \Z^p(G,M)/\B^p(G,M)
\end{multline*}
(where $d_0$ is interepreted as $0$).

The group of cochains $\C^p(G,M) = M^{G^p}$ may naturally be endowed with the product topology arising from the topology of $M$.  (Note that if $G$ is infinite and $p\geq 1$, then this will not be discrete even if $M$ is a discrete module.) For this topology, one sees easily that each map $d_p$ is continuous, so $\ker d_p$ is closed.  However, $\rm{img}\,d_p$ need not be: that is, there may be cocycles $G^p\to M$ that can be approximated arbitrarily well by coboundaries, but are not coboundaries themselves.

In some applications to geometry and dynamics, the non-closure of $\rm{img}\,d_p$ is itself an important feature of the action $\pi$.  For this reason, one also defines the \textbf{reduced cohomology groups} by
\[\rmH^p_{\rm{red}}(G,M) := \Z^p(G,M)\big/\ol{\B^p(G,M)}.\]
The cohomology of this group and module is \textbf{reduced} in degree $p$ if $\B^p(G,M)$ is closed, and hence $\rmH^p(G,M) = \rmH^p_{\rm{red}}(G,M)$.

In applications of reduced cohomology, $M$ is most often an infinite-dimensional Banach or Fr\'echet space.  Most obviously, reduced cohomology in degree one provides an important characterization of Property (T)~(\cite[Section 2.12 and Chapter 3]{BekdelaHVal08}) and some generalizations of it to other Banach modules~(\cite{BadFurGelMon07}). The purpose of this note is to prove that cohomology is always reduced for finite-dimensional vector-space modules.  It will turn out that there is no extra difficulty in allowing vector spaces over arbitrary topological fields, including $\bbR$, $\bbC$, local fields, or any other field with its discrete topology.

\begin{thm}\label{thm:main}
Let $\bbK$ be a topological field. If $G$ is a countable discrete group and $V$ is a finite-dimensional $\bbK$-linear representation of $G$, then $\B^p(G,V)$ is a closed subgroup of $\Z^p(G,V)$ for all $p \in \bbN$.
\end{thm}

It will be clear from the proof that the same result holds for many other natural choices of resolution used to compute $\rmH^\ast$, such as the homogeneous bar resolution.  Indeed, in some special cases the right choice of resolution gives an immediate proof of Theorem~\ref{thm:main}. Most simply, the conclusion is clear if there is a resolution of finite type, since then all the cohomology groups are finite-dimensional; for instance, this holds for any finite-dimensional representation of a group of type FP$_\infty$ (see, for instance,~\cite[Sections VIII.4 and VIII.5]{Bro82}).  But the general argument below reaches beyond such cases.

The proof of Theorem~\ref{thm:main} is motivated by the analogous result for compact coefficient modules, which is very simple: if $M$ is a compact Abelian $G$-module, then $\C^p(G,M) = M^{G^p}$ is also compact for every $p$, and hence the continuous image $\B^p(G,M) := d_{p-1}(\C^{p-1}(G,M))$ is necessarily closed.

To adapt this argument, we will prove a similar automatic-closure result for images of certain topological vector spaces over $\bbK$, which may be of interest in its own right.  To formulate this, we begin with the following definition.

\begin{dfn}
A \textbf{pro-f.-.d. space} over $\bbK$ is a topological $\bbK$-vector space $V$ which is an inverse limit of an inverse sequence
\[\cdots \stackrel{q_m}{\to} V_m \stackrel{q_{m-1}}{\to} \cdots \stackrel{q_1}{\to} V_1\]
in which every $V_i$ is a finite-dimensional $\bbK$-vector space (with the Hausdorff topology given by identifying with $\bbK^{\dim V_i}$), every $q_i$ is $\bbK$-linear, and where this inverse limit is given the inverse limit topology.
\end{dfn}

More concretely, this means that $V$ can be identified with the subgroup of those sequences $(v_i)_i \in \prod_i V_i$ which are \textbf{consistent}, in that $q_i(v_{i+1}) = v_i$ for every $i$.  For this subgroup, addition and scalar multiplication are pointwise, as in $\prod_i V_i$, and the topology is restricted from the product topology on $\prod_i V_i$.  We will often denote such an inverse sequence by $(V_i,q^j_i)_{j\geq i \geq 1}$, where $q^j_i := q_i\circ q_{i+1}\circ \cdots \circ q_{j-1}$ if $j > i$ and $q^i_i := \rm{id}_{V_i}$, and we will often denote a consistent sequence by $\lim_{i\leftarrow} v_i$.  The inverse limit $V$ is always equipped with a consistent family of continuous, $\bbK$-linear quotient maps $Q_i:V\to V_i$.

Of course, any finite-dimensional $\bbK$-vector space may be interpreted as a pro-f.-d. space by letting $V_i := V$ and $q_i := \rm{id}_{V_i}$ for every $i$.

The result at the heart of this paper is the following.

\begin{thm}[Automatic closure theorem]\label{thm:autoclos}
If $T:V\to W$ is a continuous $\bbK$-linear transformations of pro-f.-d. $\bbK$-vector spaces, then its image is a closed subspace of $W$.
\end{thm}

This will be proved following a sequence of lemmas.

\begin{lem}
If $V$ is the pro-f.-d. space given by the inverse sequence $(V_i,q^j_i)_{j\geq i\geq 1}$, and $Q_i:V\to V_i$ are the resulting quotient maps, then a vector subspace $W \leq V$ is closed if and only if
\[W = \bigcap_{i\geq 1}Q_i^{-1}(Q_i(W)).\]
\end{lem}

\begin{proof}
For each $i$, $Q_i(W)$ a subspace of the finite-dimensional space $V_i$, so it is closed.  Therefore the continuous pre-images $Q_i^{-1}(Q_i(W))$ are all closed, and hence so is their intersection.

On the other hand, if $W$ is closed and $v \in V\setminus W$, then $v$ has an open neighbourhood disjoint from $W$.  By definition of the inverse-limit topology, this implies that there is some $i\geq 1$ for which $Q_i(v)$ has an open neighbourhood $U_i$ such that $Q_i^{-1}(U_i) \cap W = \emptyset$.  This now implies $Q_i(v) \not\in Q_i(W)$.
\end{proof}

\begin{lem}\label{lem:comm-diag}
Suppose that $V$ and $W$ are pro-f.-d. spaces over $\bbK$, obtained from the inverse sequences $(V_i,q^j_i)_{j\geq i \geq 1}$ and $(W_i,r^j_i)_{j\geq i\geq 1}$ respectively.  Suppose also that $T:V\to W$ is a continuous $\bbK$-linear transformation.  Then there are a sequence $i_1 \leq i_2 \leq \cdots$ of indices and a sequence of $\bbK$-linear transformations $T_j:V_{i_j}\to W_j$ for which the following diagram commutes:
\begin{center}
$\phantom{i}$\xymatrix{
V \ar@{->>}[r]\ar_T[d] & \cdots \ar@{->>}[r] & V_{i_3} \ar@{->>}[r]\ar^{T_3}[d] & V_{i_2} \ar@{->>}[r]\ar^{T_2}[d] & V_{i_1} \ar^{T_1}[d]\\
W \ar@{->>}[r] & \cdots \ar@{->>}[r] & W_3 \ar@{->>}[r] & W_2 \ar@{->>}[r] & W_1
}
\end{center}
\end{lem}

\begin{proof}
Let $Q_i:V\to V_i$ and $R_j:W\to W_j$ be the quotient maps.

If an $i_j$ and $T_j$ exist for some $j$, then the choice of $T_j$ for this $i_j$ is clearly unique.  It therefore suffices to prove existence for any fixed $j$, since that uniqueness then forces the resulting family to be consistent, so that the above diagram commutes.

We must therefore show that, for each $j$, the composition $R_j\circ T:V\to W_j$ factorizes through $Q_i$ for some $i$.  Effectively, this restricts our attention to a finite-dimensional target space $W$.  After this restriction, it suffices to show that some $i$ satisfies $\ker Q_i \leq \ker T$.

Since $\ker T$ is a closed vector subspace of $V$, the previous lemma gives
\[\ker T = \bigcap_{i\geq 1}Q_i^{-1}(Q_i(\ker T)).\]
However, $\ker T$ has finite co-dimension in $V$, so this finite intersection must stabilize at some finite $i$, giving $\ker T = Q_i^{-1}(Q_i(\ker T)) \geq \ker Q_i$, as required.
\end{proof}

In the setting of Lemma~\ref{lem:comm-diag}, the space $V$ may always be identified with the inverse limit of the cofinal inverse subsequence
\[\cdots \stackrel{q^{i_4}_{i_3}}{\to} V_{i_3} \stackrel{q^{i_3}_{i_2}}{\to} V_{i_2} \stackrel{q^{i_2}_{i_1}}{\to} V_{i_1}.\]
Henceforth, after fixing a continuous linear transformation $T$ of interest, we will usually make this identification, and accordingly relabel this new inverse subsequence as $(V_i,q^j_i)_{j\geq i\geq 1}$.

\begin{cor}
In the notation of Lemma~\ref{lem:comm-diag}, and after re-labeling as above,
\[\ker T = \bigcap_{i \geq i}Q_i^{-1}(\ker T_i).\]
\qed
\end{cor}

In spite of this corollary, one can certainly have $\ker T_i \gneqq Q_i(\ker T)$ for all finite $i$.  However, the following more subtle conclusion does always hold.

\begin{lem}
With $T$ and its presentation as above, there is a non-decreasing function $\ell:\bbN\to \bbN$ such that $\ell(i) \geq i$ and
\[Q_i(\ker T) = q^\ell_i(\ker T_\ell) \quad \forall i \in \bbN\ \forall \ell \geq \ell(i).\]
\end{lem}

\begin{proof}
For fixed $i$, $(q^\ell_i(\ker T_\ell))_{\ell \geq i}$ is a decreasing sequence of subspaces of the finite-dimensional space $V_i$, hence must stabilize: that is, there is some $\ell \geq i$ such that
\begin{eqnarray}\label{eq:stab}
q^j_i(\ker T_j) = q^\ell_i(\ker T_\ell) \quad \forall j\geq \ell.
\end{eqnarray}
Let $\ell(i)$ be a choice of such an $\ell$, and assume without loss of generality that $\ell(i) > 1$.  We will show that this common image is equal to $Q_i(\ker T)$.  By re-labeling, it suffices to prove this when $i=1$.

Define a subsequence of indices $i_0 < i_1 < i_2 < \ldots$ by letting $i_0 := 1$ and recursively setting $i_{m+1} := \ell(i_m)$.

Suppose that $v \in q_1^{i_1}(\ker T_{i_1})$.  By~(\ref{eq:stab}), $v$ also lies in $q_1^{i_2}(\ker T_{i_2})$, since $i_2 > i_1$, so suppose $u \in \ker T_{i_2}$ satisfies $q_1^{i_2}(u) = v$, and now let $v_1 := q_{i_1}^{i_2}(u) \in q^{i_2}_{i_1}(\ker T_{i_2})$, so we still have $v = q^{i_1}_1(v_1)$.

Now, repeating this reasoning from~(\ref{eq:stab}) with $i_1$ in place of $i_0$ and  $v_1$ in place of $v$, we obtain $v_2 \in q^{i_3}_{i_2}(\ker T_{i_3})$ such that $q^{i_2}_{i_1}(v_2) = v_1$.  Continuing in this way gives a consistent sequence $(v_j)_{j \geq 1}$ in $\prod_{j\geq 1}V_{i_j}$ such that $v_{i_j} \in q^{i_{j+1}}_{i_j}(\ker T_{i_{j+1}})$ for every $j$.  Therefore $\lim_{i\leftarrow}v_i$ lies in $\bigcap_{j\geq 1}Q_{i_j}^{-1}(\ker T_{i_j}) = \ker T$, and its image in $V_1$ is $v$, as required.
\end{proof}

\begin{proof}[Proof of Theorem~\ref{thm:autoclos}]
Let $T:V\to W$, and suppose that $w \in W$ is such that $w_i := R_i(w) \in R_i(T(V))$ for every $i$.  We need to show that $w \in T(V)$.

More explicitly, we need to find a consistent sequence $(v_i)_i$ in $\prod_i V_i$ such that $T_i(v_i) = w_i$ for each $i$.  The sequence $(v_i)_i$ will be constructed recursively.  The key idea is that we construct in parallel a sequence $(\t{v}_i)_i \in \prod_i V_{\ell(i)}$ such that $v_i = q^{\ell(i)}_i(\t{v}_i)$ for each $i$.

To begin, let $\t{v}_1 \in V_{\ell(1)}$ be a $T_{\ell(1)}$-pre-image of $w_{\ell(1)}$, and let $v_1 := q^{\ell(1)}_1(\t{v}_1)$.

Now, suppose that for some $i \geq 2$ we have already constructed $(\t{v}_j)_{j=1}^{i-1} \in \prod_{j=1}^{i-1}V_{\ell(j)}$ and $(v_j)_{j=1}^{i-1} \in \prod_{j=1}^{i-1}V_j$ such that \begin{itemize}
\item $T_{\ell(j)}(\t{v}_j) = w_{\ell(j)}$,
\item $v_j = q^{\ell(j)}_j(\t{v}_j)$,
\item and $q^j_{j'}(v_j) = v_{j'}$ whenever $j' \leq j \leq i-1$ (note that we do not require the analog of this consistency for $(\t{v}_j)_j$).
\end{itemize}
We will construct $v_i$ and $\t{v}_i$ so that all of these properties still hold for the enlarged family.

To do this, let $v'$ be any element of $T_{\ell(i)}^{-1}\{w_{\ell(i)}\}$.  Then we have
\begin{eqnarray*}
&&T_{\ell(i-1)}q^{\ell(i)}_{\ell(i-1)}(v') = r^{\ell(i)}_{\ell(i-1)}T_{\ell(i)}(v') = r^{\ell(i)}_{\ell(i-1)}(w_{\ell(i)}) = w_{\ell(i-1)} = T_{\ell(i-1)}(\t{v}_{i-1})\\
&&\Longrightarrow \quad \t{v}_{i-1} - q^{\ell(i)}_{\ell(i-1)}(v') \in \ker T_{\ell(i-1)}\\
&&\Longrightarrow \quad q^{\ell(i-1)}_{i-1} (\t{v}_{i-1}) - q^{\ell(i)}_{i-1}(v') = v_{i-1} - q^{\ell(i)}_{i-1}(v') \in q^{\ell(i-1)}_{i-1}(\ker T_{\ell(i-1)}).
\end{eqnarray*}
By the definition of $\ell(\cdot)$, this last subspace is equal to $q^j_{i-1}(\ker T_j)$ for any other $j \geq \ell(i-1)$; in particular, it is equal to $q^{\ell(i)}_{i-1}(\ker T_{\ell(i)})$.  Let $v'' \in \ker T_{\ell(i)}$ be such that
\[q^{\ell(i)}_{i-1}(v'') = v_{i-1} - q^{\ell(i)}_{i-1}(v') \quad \Longrightarrow \quad q^{\ell(i)}_{i-1}(v' + v'') = v_{i-1}.\]

Finally, letting $\t{v}_i := v' + v'' \in V_{\ell(i)}$ and $v_i := q^{\ell(i)}_i(\t{v}_i)$, one has
\[T_{\ell(i)}(\t{v}_i) = T_{\ell(i)}(v' + v'') = T_{\ell(i)}(v') = w_{\ell(i)}\]
and
\[q^i_{i-1}(v_i) = q^{\ell(i)}_{i-1}(v' + v'') = v_{i-1},\]
so the induction continues.  This completes the proof.
\end{proof}

\begin{proof}[Proof of Theorem~\ref{thm:main}]
For each $p$, the group $\C^p(G,V) = V^{G^p}$ is a countable product of finite-dimensional $\bbK$-vector-spaces, so it is a pro-f.-d. space.  Since the boundary homomorphisms
\[d_{p+1}:\C^p(G,V)\to \C^{p+1}(G,V)\]
are clearly continuous (as they are for cohomology with values in any topological module), by Theorem~\ref{thm:autoclos} they are also closed: that is,
\[\B^p(G,V) = d_p(\C^{p-1}(G,V))\]
is closed in $\Z^p(G,V)$, as required.
\end{proof}

Since a countable product of pro-f.-d. spaces with the product topology is still pro-f.-d., the above proof actually gives the following strengthening of Theorem~\ref{thm:main}.

\begin{cor}
If $G$ is a countable discrete group, $\bbK$ is a topological field and $V$ is a pro-f.-d. $G$-space over $\bbK$, then $\B^p(G,V)$ is closed in $\Z^p(G,V)$ for all $p\geq 0$. \qed
\end{cor}

However, the above methods do not seem to bear on the following natural relative of Theorem~\ref{thm:main}.

\begin{ques}\label{ques:1}
Are there a countable group $G$ and an integer $p\geq 1$ such that the cohomology $\rmH^p(G,\bbZ)$ is not reduced, where $\bbZ$ is given the trivial $G$-action?  Can this occur for other discrete $G$-modules?
\end{ques}

This was Question 7.4 in~\cite{AusMoo--cohomcty}. Of course, for $M$ with trivial $G$-action one always has $\B^1(G,M) = 0$, so the interesting cases start at $p=2$.  Moreover, whenever $M$ is a Polish $G$-module and one knows that $\rmH^p(G,M)$ is countable for some $p$, it follows that $\B^p(G,M)$ is closed for very abstract reasons.  Indeed, $\B^p(G,M) = d_p(\C^{p-1}(G,M))$ is a continuous image of a Polish space, hence analytic in the sense of Souslin.  If $\Z^p(G,M)\setminus \B^p(G,M)$ is a countable union of cosets of $\B^p(G,M)$, then one has that $\B^p(G,M)$ is also co-analytic, hence Borel.  As such it is either meagre in $\Z^p(G,M)$, or co-meagre inside some nonempty open set.  Using again that countably many cosets of $\B^p(G,M)$ cover $\Z^p(G,M)$, it cannot be meager, so it is co-meager in some non-empty open subset $U \subseteq \Z^p(G,M)$.  However, this now implies that \[\B^p(G,M) \supseteq (U\cap \B^p(G,M)) - (U\cap \B^p(G,M)),\]
which contains a neighbourhood of $0$, so $\B^p(G,M)$ is open and therefore also closed.

This last argument applies, in particular, if $G$ satisfies some cohomological finiteness properties, such as FP$_k$ for some $k > p$; see, for instance,~\cite[Chapter VIII]{Bro82}.

One cannot hope to answer Question~\ref{ques:1} positively using a method too close to that of the present note, because the analog of Theorem~\ref{thm:autoclos} is false for inverse limits of finitely-generated torsion-free Abelian groups.

\begin{ex}\label{ex:1}
Let $W := V := \bbZ^\bbN$, the infinite product with the product topology, and let $T:V\to W$ be given by
\[T(p_1,p_2,p_3,\ldots) := (p_1 - 2p_2,p_2 - 2p_3,p_3 - 2p_4,\ldots).\]
Each of $V$ and $W$ is an inverse limit for the sequence of projections $\bbZ^{n+1}\to \bbZ^n$ onto the initial coordinates, and $T$ may be described using these projections via a commutative diagram as in Lemma~\ref{lem:comm-diag}, since each coordinate output by $T$ depends on only two input coordinates.  It also follows that $T$ is continuous.

A point $(q_1,q_2,\ldots) \in \bbZ^\bbN$ lies in $T(V)$ if and only if there is some $(p_1,p_2,\ldots) \in \bbZ^\bbN$ such that
\begin{eqnarray*}
q_1 = p_1 - 2p_2, \quad q_2 = p_2 - 2p_3, \quad \ldots.
\end{eqnarray*}
This re-arranges to give
\[p_2 = \frac{1}{2}(p_1 - q_1), \quad p_3 = \frac{1}{4}(p_1 - q_1 - 2q_2), \quad p_4 = \frac{1}{8}(p_1 - q_1 - 2q_2 - 4q_3), \quad \ldots,\]
and this requires, in particular, that
\[p_1 \in \bigcap_{i\geq 1}(q_1 + 2q_2 + \cdots + 2^{i-1}q_i + 2^i\bbZ).\]
If one can find such an integer $p_1$, then the equations above allow one to find suitable $p_j$ for all $j\geq 1$.  Such a $p_1$ exists, for example, if $q_i = 0$ for all large enough $i$, since in that case one can let $p_1 := \sum_{i\geq 1}2^{i-1}q_i$.  However, it does not exist for $(q_1,q_2,\ldots) = (1,0,1,0,\ldots)$, since then one has
\[\min\{|p|\,|\ p \in q_1 + 2q_2 + \cdots + 2^{i-1}q_i + 2^i\bbZ\}\sim \frac{1}{3}2^i \to \infty \quad \hbox{as}\ i\to\infty.\]
Therefore the image $T(V)$ is dense in $W$, but not equal to the whole of $W$.  \fin
\end{ex}

It could be interesting to try to adapt the idea behind Example~\ref{ex:1} into an example in which $\B^p(G,\bbZ)$ is not closed for some $G$ and $p$.  I do not know how to do this, but simpler examples show that in the `mixed' setting of actions on torsion-free Abelian Lie groups, the analog of Theorem~\ref{thm:main} is itself false.  The following is essentially copied from the example in~\cite[Lemma 7.7]{AusMoo--cohomcty}.

\begin{ex}\label{ex:2}
Let $\a,\b \in \bbR$ be irrational and rationally independent, let $M = \bbR\times \bbZ^2$, and give it the $\bbZ$-action generated by
\[T(t,p,q) := (t + p\a + q\b,p,q).\]
Then an easy calculation gives that $\Z^1(\bbZ,M) \cong \rm{Hom}(\bbZ,\bbR) = \bbR$, and under this isomorphism $\B^1(\bbZ,M)$ is identified with the submodule $\bbZ\a + \bbZ\b$, which is dense but not closed in $\bbR$. \fin
\end{ex}

Another obvious extension of Theorem~\ref{thm:main} would be to non-discrete groups $G$.  In that setting there are several cohomology theories for Polish $G$-modules that could be of interest, such as $\rmH^\ast_\rm{cts}$, defined using continuous cochains, and $\rmH^\ast_\rm{m}$, defined using measurable cochains.  In case the coefficient module is a Fr\'echet space, it was shown in~\cite[Theorem A]{AusMoo--cohomcty} that these theories coincide.

\begin{ques}\label{ques:2}
If $G$ is a locally compact, second countable group, is it true that $\rmH_{\rm{cts}}^p(G,V)$ is reduced for every $p$ and every Euclidean $G$-space $V$?  Is it true that $\rmH^p_\rm{m}(G,\bbZ)$ is reduced for every $p$?
\end{ques}

The paper~\cite{AusMoo--cohomcty} offers a more complete overview of various cohomology theories in this setting, and a more thorough set of further references.

The issue of Question~\ref{ques:2} is potentially important for calculational problems in which one wishes to use a version of the Lyndon-Hochschild-Serre spectral sequence.  If $H\unlhd G \twoheadrightarrow K$ is a short exact sequence of locally compact, second countable groups, and $V$ is a Fr\'echet $G$-module, then in principle this spectral sequence gives
\[\rmH^p_{\rm{cts}}(H,\rmH^q_\rm{cts}(K,V))\Longrightarrow \rmH^{p+q}_{\rm{cts}}(G,V).\]
However, for this to make sense one must consider continuous cochains $H^p\to \rmH^q_{\rm{cts}}(K,V)$ for the quotient topology on $\rmH^q_{\rm{cts}}(K,V)$, and the calculation works out correctly only if that quotient topology is Hausdorff. For this reason, an affirmative answer to Question~\ref{ques:2} could enlarge the known domain of applicability of this spectral sequence.  This issue is discussed in more detail for continuous cohomology into Fr\'echet modules in~\cite[Chapter IX]{BorWal00}, and for the related setting of measurable cohomology into Polish modules in~\cite{Moo76(gr-cohomIII)}.  It also appears in~\cite{Blanc79}, which uses a comparison between $\rmH^\ast_{\rm{cts}}$ and a kind of $L^p_\rm{loc}$-cohomology to justify some cases of the spectral sequence.

Question~\ref{ques:2} does not fall within the methods of the present paper, because if $V$ is Euclidean but $G$ is non-discrete then the groups of continuous cochains $\C^p_{\rm{cts}}(G,V)$ are not pro-f.-d. in the topology of locally uniform convergence.  If Question~\ref{ques:2} has an affirmative answer, it will presumably require a more specific analysis of the closure properties of the coboundary operators in this setting.  Indeed, the paper~\cite{AusMoo--cohomcty} already contains examples of such arguments: see~\cite[Theorem D]{AusMoo--cohomcty}, which answers both parts of Question~\ref{ques:2} affirmatively in case $G$ is almost connected (that is, its identity component is co-compact).  The arguments there use comparison results between different cohomology theories, the Gleason-Montgomery-Zippin theorem, and the Lyndon-Hochschild-Serre spectral sequence.  I expect that similar tools can answer Question~\ref{ques:2} affirmatively for totally disconnected groups, since these may be written as compact extensions of discrete groups, and we know the desired result for both of those classes separately (by~\cite{AusMoo--cohomcty} and the present paper, respectively). But the fully general case seems more subtle.

\subsection*{Acknowledgement}

I am grateful to David Fisher for helpful discussions and suggestions, and to the anonymous referee for providing some additional context and references.

\bibliographystyle{abbrv}
\bibliography{bibfile}

\parskip 0pt
\parindent 0pt

\vspace{7pt}

\small{\textsc{Courant Institute of Mathematical Sciences, New York University, 251 Mercer St, New York NY 10012, USA}

\vspace{7pt}

Email: \verb|tim@cims.nyu.edu|

URL: \verb|cims.nyu.edu/~tim|}

\end{document}